\documentclass[11pt, a4paper]{amsart}

\usepackage{amsfonts,amsmath,amssymb, amscd}
\usepackage{txfonts}
\usepackage[all]{xy}

\newtheorem{theorem}{Theorem}[section]
\newtheorem{lemma}[theorem]{Lemma}

\newtheorem{prop}[theorem]{Proposition}

\theoremstyle{definition}
\newtheorem{rem}[theorem]{Remark}

\newtheorem{Pb}[theorem]{Problem}
\newtheorem{exa}[theorem]{Example}

\newcommand\pf{\begin{proof}}
\newcommand\epf{\end{proof}}

\newcommand\card{\mathrm{card}}
\newcommand\ab{\mathrm{ab}}
\newcommand\co{\mathrm{co}}

\newcommand\eps{\varepsilon}
\newcommand\Frac{\operatorname{Frac}}

\newcommand\gr{\operatorname{gr}}

\renewcommand\AA{\mathcal A}
\newcommand\BB{\mathcal B}
\newcommand\CC{\mathcal C}
\newcommand\UU{\mathcal U}
\newcommand\VV{\mathcal V}
\newcommand\WW{\mathcal W}

\newcommand\NN{\mathbb{N}}
\newcommand\ZZ{\mathbb{Z}}
\newcommand\QQ{\mathbb{Q}}
\newcommand\mm{\mathfrak{m}}

\DeclareMathOperator{\id}{id}

\newcommand{\rbtimess}{\mathbin{\raisebox{0.2ex}%
{\makebox[1em][l]{${\scriptstyle>\mathrel{\mkern-4mu}\blacktriangleleft}$}}}}
\newcommand{\rdtimess}{\mathbin{\raisebox{0.2ex}%
{\makebox[1em][l]{${\scriptstyle>\mathrel{\mkern-4mu}\rlap{$\scriptstyle{\lessdot}$}{\vartriangleleft}}$}}}}

\numberwithin{equation}{section}

\title[The Noether problem for Hopf algebras]
{The Noether problem for Hopf algebras}

\author{Christian Kassel}
\address{Christian Kassel, 
Institut de Recherche Math\'e\-ma\-tique Avanc\'ee,
CNRS \& Universit\'e de Strasbourg,
7 rue Ren\'{e} Descartes, 67084 Strasbourg, France}
\email{kassel@math.unistra.fr}

\author{Akira Masuoka}
\address{Akira Masuoka,
Institute of Mathematics, 
University of Tsukuba, 
Ibaraki 305-8571, Japan}
\email{akira@math.tsukuba.ac.jp}

\begin{document}

\begin{abstract}
In previous work, Eli Aljadeff and the first-named author 
attached an algebra~$\BB_H$ of rational fractions to each Hopf algebra~$H$.
The generalized Noether problem is the following: for which finite-dimensional Hopf algebras~$H$
is $\BB_H$ the localization of a polynomial algebra?
A positive answer to this question when $H$ is the algebra of functions on a finite group~$G$
implies a positive answer for the classical Noether problem for~$G$.
We show that the generalized Noether problem has a positive answer 
for all finite-dimensional pointed Hopf algebras over a field of characteristic zero
(we actually give a precise description of~$B_H$ for such a Hopf algebra).

A theory of polynomial identities for comodule algebras over a Hopf algebra~$H$
gives rise to a universal comodule algebra 
whose subalgebra of coinvariants~$\VV_H$ maps injectively into~$\BB_H$.
In the second half of this paper, we show that $\BB_H$ is a localization of~$\VV_H$
when $H$ is a finite-dimensional pointed Hopf algebra in characteristic zero.
We also report on a result by Uma Iyer 
showing that the same localization result holds when $H$ is the algebra of functions on a finite group.
\end{abstract}

\maketitle


(Primary);
13A50, 
13B30, 
12F20 
(Secondary)

\hspace{3cm}

\section{\large Introduction}
Let $G$ be a finite group and $k$ a field. Consider the purely transcendental extension $K = k(t_g\, | \, g\in G)$ of~$k$
generated by indeterminates~$t_g$ indexed by the elements of~$G$. The group~$G$ acts on~$K$ by
left multiplication: $h\cdot t_g = t_{hg}$ ($g,h \in G$). 
Let $L = K^G$ be the subfield of $G$-invariant elements of~$K$.
Emmy Noether posed the following problem in\,\cite{No}: 
\emph{is $L$ a purely transcendental extension of~$k$?}
A positive answer ensures that $G$ can be realized as the Galois group of a Galois extension of~$k$;
it also implies the existence of a generic polynomial for~$G$ over~$k$ (at least when $k$ is infinite). 

There is an abundant bibliography on Noether's problem. Answers depend on the group and on the base field. 
For instance, by Fischer\,\cite{Fi} the extension $L/k$ is purely transcendental 
if $G$ is abelian and $k$~contains a primitive $e$-th root of unity, where $e$ is the exponent of~$G$.
If $k$ has not enough roots of unity, for instance when $k = \QQ$ is the field of rationals, 
then Swan\,\cite{Swa} showed that there are cyclic groups~$G$ such that $L$ is not purely transcendental over~$k$. 
For non-abelian groups, the answer to the problem may be negative even over an algebraically closed field: 
by Saltman\,\cite{Sa} this is the case for certain meta-abelian $p$-groups when $k$ is the field of complex numbers.

In this note we extend Noether's problem to the framework of finite-dimensional Hopf algebras
over the base field~$k$. 
To such a Hopf algebra~$H$ Eli Aljadeff and the first-named author associated an algebra~$\BB_H$ of
rational fractions, which is a finitely generated smooth domain of Krull dimension equal to the dimension of~$H$
(see\,\cite{AK}).
The algebra~$\BB_H$ is called the \emph{generic base algebra} associated to~$H$; 
in the terminology of non-commutative geometry it is the ``base space'' of a ``non-commutative fiber bundle'' whose 
fibers are the forms of~$H$.

The generalized Noether problem (GNP) which we introduce in this paper is the following: 
\emph{is $\BB_H$ the localization of a polynomial algebra?}
A positive answer to~(GNP) for the Hopf algebra of $k$-valued functions on
a finite group~$G$ implies a positive answer for the classical Noether problem for~$G$ and~$k$.
Our first main result (Theorem\,\ref{th-GNP}) states that (GNP) has a positive answer 
for all \emph{finite-dimensional pointed} Hopf algebras over a field of characteristic zero.
Actually in Theorem\,\ref{th-GNP-degrees} we prove more precisely that for such a Hopf algebra~$H$, 
\[
\BB_H = k[u_1^{\pm 1}, \ldots, u_{\ell}^{\pm 1}, u_{\ell+1}, \ldots, u_n]  \, ,
\]
where $n$ is the dimension of~$H$ and $\ell$ is the order of the group~$G$ of group-like elements of~$H$,
and where $u_1, \ldots, u_n$ are monomials whose degrees are bounded by an integer defined 
in terms of a certain abelian quotient of~$G$.
The latter statement is a Hopf algebra analogue of the fact, due to Noether\,\cite{No0}, 
that the algebra of $G$-invariant polynomials in~$k[t_g\, | \, g\in G]$ 
is generated by homogeneous polynomials of degree~$\leq \card\, G$.

A theory of polynomial identities for comodule algebras had also been set up in\,\cite{AK}, giving rise to the
so-called {universal comodule algebra}~$\UU_H$, an analogue of the relatively free algebra
of the classical theory of polynomial identities. 
The subalgebra~$\VV_H$ of $H$-coinvariants of~$\UU_H$ maps injectively into~$\BB_H$.
In Section\,\ref{sec-PI} we show that $\BB_H$ is a localization of~$\VV_H$
when $H$ is a finite-dimensional pointed Hopf algebra over a field of characteristic zero
(see Theorem\,\ref{th-local}).
Uma Iyer showed likewise that $\BB_H$ is a localization of~$\VV_H$
when $H$ is the algebra of $k$-valued functions on a finite group
whose order is prime to the characteristic of~$k$;  
with her permission we state and prove her result (Theorem\,\ref{thm-UI}) in Section\,\ref{subsec-UI}.

Throughout the paper we fix a field~$k$.
All linear maps are to be $k$-linear
and unadorned tensor products mean tensor products over~$k$.

\section{\large Pointed Hopf algebras}\label{sec-general}

In this section we prove some facts on pointed Hopf algebras
needed in the proofs of the main results.

\subsection{Hopf algebras and comodule algebras}\label{subsec-Hopf}

By algebra we mean an associative unital $k$-algebra
and by coalgebra a coassociative counital $k$-coalgebra. 
We denote the coproduct of a coalgebra by~$\Delta$ and its counit by~$\eps$.
We shall also make use of a Heyneman-Sweedler-type notation 
for the image 
\[
\Delta(x) = x_1 \otimes x_2
\]
of an element~$x$ of a coalgebra~$C$ under the coproduct, and we write
$$\Delta^{(2)}(x) = x_1 \otimes x_2 \otimes x_3$$
for the image of~$x$ under the iterated coproduct 
$\Delta^{(2)} = (\Delta \otimes \id_C) \circ \Delta$. 

Given a Hopf algebra~$H$, we denote its counit by~$\eps$ and its antipode by~$S$.
We also denote the augmentation ideal~$\ker(\eps : H \to k)$ by~$H^+$
and the group of group-like elements of~$H$ by~$G(H)$.

Recall that a right $H$-\emph{comodule algebra} over a Hopf algebra~$H$ is an algebra~$A$ 
equipped with a right $H$-comodule structure whose (coassociative, counital) \emph{coaction}
$\delta : A \to A \otimes H$ is an algebra map.
The subalgebra~$A^{\co-H}$ of \emph{right coin\-var\-iants} of an $H$-comodule algebra~$A$
is the following subalgebra of~$A$:
\begin{equation*}
A^{\co-H} = \{ a \in A \, | \, \delta(a)  = a \otimes 1\} \, .
\end{equation*}

\subsection{Finite-dimensional pointed Hopf algebras}\label{subsec-pointed}

In this subsection we present three technical results for finite-dimensional pointed Hopf algebras.

Recall that a Hopf algebra~$H$ is \emph{pointed} if any simple subcoalgebra is one-dimensional.
Group algebras~$k[G]$, enveloping algebras~$U(\mathfrak{g})$ of Lie algebras,
Drinfeld--Jimbo quantum enveloping algebras~$U_q(\mathfrak{g})$ 
and their finite-dimensional quotients~$u_q(\mathfrak{g})$
are important examples of pointed Hopf algebras.

See\,\cite[Chap.\,VIII]{Swe} and\,\cite[Chap.\,5]{Mo} for basic properties of pointed Hopf algebras.

\begin{lemma}\label{lem:com}
Any finite-dimensional commutative pointed Hopf algebra over a field of characteristic zero is a group algebra.
\end{lemma}

\begin{proof}
Since $H$ is pointed, by scalar extension to the algebraic closure~$\bar{k}$ of~$k$, 
we may suppose that $k$ is algebraically closed;  note that $G(H) = G(H\otimes \bar{k})$. 
It then follows from\,\cite[Th.\,13.1.2]{Swa} that $H$ is reduced, 
hence $H = O_k(G)$ is the Hopf algebra of $k$-valued functions on some finite group~$G$. 
For the dual Hopf algebra~$H^*=k[G]$, the pointedness of~$H$ means that any
simple $H^*$-module is one-dimensional, which is only possible if $G$ is abelian.
By Fourier transform, we have $H = k[\widehat{G}]$, where $\widehat{G}$ is the group of characters of~$G$.
\end{proof}

Given a Hopf algebra~$H$, let~$H_{\ab}$ be the largest commutative Hopf algebra quotient of~$H$
and let $q: H \to H_{\ab}$ be the canonical Hopf algebra surjection.

\begin{lemma}\label{lemma:retraction}
Let $H$ be a finite-dimensional pointed Hopf algebra. Assume that $H_{\ab}$ is a group algebra. 
Then there exists a right $k[G(H)]$-module coalgebra retraction $\gamma : H \to k[G(H)]$ of the 
inclusion $k[G(H)] \subset H$ such that the composite
\[
H \overset{\gamma}{\longrightarrow} k[G(H)] \overset{q |_{k[G(H)]}}{\longrightarrow} H_{\ab}
\]
coincides with~$q$. 
\end{lemma}

By Lemma\,\ref{lem:com} and the fact that any Hopf algebra quotient of a pointed Hopf algebra is pointed,
the assumption in the previous lemma is always satisfied when the base field is of characteristic zero. 

\begin{proof}
Set $G = G(H)$. The quotient $R = H/H(k[G])^+$ is a quotient coalgebra of~$H$.  
Let $H \to R \, ; h \mapsto \bar{h}$ denote the quotient coalgebra map. 
By the cosemisimplicy of $k[G]$ it follows from \cite[Lemma~4.2]{Ma3} (or \cite[Cor.~3.11]{Ma4})
that there is a right $k[G]$-module coalgebra retraction $\gamma : H \to k[G]$ of the 
inclusion.  Then the map
\begin{equation}\label{lambda}
\lambda : R \to k[G] \otimes R \, ; \quad \lambda(\bar{h}) 
= \gamma(h_{1}) \, S(\gamma(h_{3})) \otimes \bar{h}_{2} 
\end{equation}
defines a left $k[G]$-comodule coalgebra structure on~$R$,
and the associated smash-coproduct coalgebra
$R \rbtimess k[G]$ is isomorphic to $H$ via the unit-pre\-serv\-ing, right $k[G]$-module coalgebra isomorphism
\begin{equation}\label{xi} 
\xi : H \overset{\cong}{\longrightarrow} R \rbtimess k[G] \, ; \quad \xi (h) = \bar{h}_{1}\otimes \gamma(h_{2}) \, . 
\end{equation}
Note that this kind of isomorphism from $H$ to a smash-coproduct of $R$ by $k[G]$ arises, in this manner, 
uniquely from a retraction $H \to k[G]$ as above. 
Identify $H$ with $R \rbtimess k[G]$ via the isomorphism\,\eqref{xi}.
Since $\gamma$ is then identified with the quotient $H \to H/R^+k[G]=k[G]$, one sees that
$q|_{k[G]}\circ \gamma = q$ if and only if 
\begin{equation}\label{vanish}
q\ \text{vanishes on}\ R^+= R^+\otimes k \, . 
\end{equation}
We will now choose a new retraction~$\gamma$ so that this last condition is satisfied. 

Denote the coradical filtrations on~$H$ and on~$R$ respectively by
\[ 
k[G] =H_0 \subset H_1 \subset H_2 \subset \dots, \qquad k =R_0 \subset R_1 \subset R_2 \subset \dots .
\]
Note that each~$H_n$ is a right $k[G]$-module subcoalgebra in~$H$, and the isomorphism\,\eqref{xi} restricts to 
\begin{equation*}
H_n \overset{\cong}{\longrightarrow} R_n \rbtimess k[G], \quad n=0,1,\dots. 
\end{equation*}
On the associated gradings, there is an induced isomorphism of graded right $k[G]$-module coalgebras. 
It coincides with the canonical isomorphism of graded Hopf algebras 
\begin{equation*}
\gr H = \bigoplus_{n \geq 0} \, H_n/H_{n-1} \overset{\cong}{\longrightarrow} \gr \, R \rdtimess k[G]\, ,
\qquad
\gr R = \bigoplus_{n \geq 0} \, R_n/R_{n-1}
\end{equation*}
onto the biproduct $\gr R \rdtimess k[G]$ (see\,\cite[Th.\,3]{Ra})
which arises from the projection $\gr H\to H_0 = k[G]$.
Indeed, the composition of the induced isomorphism with $\varepsilon\otimes \id$ coincides with the projection.
For the biproduct we naturally identify the graded coalgebra~$\gr R$ with $\gr H/(\gr H)\, k[G]^+$,
and thereby regard it as a graded Hopf algebra object in the braided category
of Yetter-Drinfeld modules over $k[G]$. 

Since $H_{\ab}$ is cosemisimple by the assumption, $q: H \to H_{\ab}$ is a filtered Hopf algebra map, 
where $H_{\ab}$ is trivially coradically filtered, meaning that $(H_{\ab})_n = H_{\ab}$ for all $n\geq 0$. 
Hence, $q$ induces a graded Hopf algebra map $\gr q : \gr H \to H_{\ab}= H_{\ab}(0)$
which vanishes in positive degrees. It follows that $q(R^+_1)=q(R^+_n)$ for all $n >0$. 
Therefore, for \eqref{vanish} to hold, it is enough for the following to hold:
\begin{equation}\label{vanish_on_R_1}
q\ \text{vanishes on}\ R_1^+ = R_1 \cap R^+.
\end{equation}

Note that $R_1^+$ consists of all primitive elements in the coalgebra~$R$. 
It has a basis $v_1,\dots,v_r$ such that
$\lambda(v_i) = g_i\otimes v_i$ for some $g_i \in G$; see\,\eqref{lambda}. 
Note that $v_i$ is $(1,g_i)$-primitive in $H$, namely $\Delta(v_i) = g_i \otimes v_i + v_i \otimes 1$ on~$H$. 
Since $q(v_i)$ is $(1,q(g_i))$-primitive in the group algebra~$H_{\ab}$, 
we have $q(v_i) = c_i \, (1 - q(g_i))$ for some $c_i \in k$. 
Define $w_i = v_i - c_i \, (1-g_i)$; this element remains $(1,g_i)$-primitive. 
Note that
\begin{equation}\label{vanish_on_w_i}
q(w_i) = 0 \, ,\quad 1 \leq i \leq r \, . 
\end{equation}
We see that $v_i \mapsto w_i$ and $1 \mapsto 1$ give rise to a unit-preserving,
right $k[G]$-module coalgebra isomorphism 
\[ 
R_1 \rbtimess k[G] \overset{\cong}{\longrightarrow} H_1 \, . 
\]
By \cite[Th.~4.1]{Ma3} (or \cite[Th.~3.10]{Ma4})
the corresponding retraction $H_1 \to k[G]$ of $k[G] \subset H_1$, 
namely the inverse of the last isomorphism composed with 
$\varepsilon \otimes \mathrm{id} : R_1 \rbtimess k[G] \to k[G]$ 
extends to a right $k[G]$-module coalgebra retraction $H \to k[G]$ of $k[G]\subset H$. 
Choose the latter retraction as the new~$\gamma$. 
Then one sees from~\eqref{vanish_on_w_i} that the desired condition~\eqref{vanish_on_R_1}
is satisfied.
\end{proof}

Let $A$ be a \emph{commutative pointed} Hopf algebra containing a left coideal subalgebra~$B$. 
We denote by~$Q$ the quotient Hopf algebra $A/(B^+)$, where $(B^+)$ is the ideal 
of~$A$
generated by $B^+ = B \cap A^+$.
The natural projection turns~$A$ into a right $Q$-comodule algebra.
Let $G(B) = G(A) \cap B$ be the monoid consisting of all group-like elements of~$A$ contained in~$B$. 

\begin{lemma}\label{lem-AB}
With the previous notation,
the left coideal subalgebra $A^{\co-Q}$ consisting of the right $Q$-coinvariant elements of~$A$
is the localization $B [ \, g^{-1} \, | \; g \in G(B) ]$ of~$B$ by~$G(B)$. 
\end{lemma}

\begin{proof}
By\,\cite[Th.\,1.3]{Ma1}, $C = A^{\co-Q}$ is the smallest left coideal subalgebra of~$A$ containing~$B$
such that $G(C) = G(A) \cap C$ is a group. 
One easily checks that $B[ \, g^{-1} \, | \, g \in G(B) ]$ is such a left coideal subalgebra. 
\end{proof}

\section{\large The generic algebra associated to a Hopf algebra}\label{sec-generic}

We now introduce the objects needed to state the generalized Noether problem 
in Section\,\ref{sec-Noether-pb}.

\subsection{The free commutative Hopf algebra generated by a coalgebra}\label{subsec-Takeuchi}

By Ta\-keu\-chi\,\cite[Chap.\,IV]{Ta}, given a coalgebra~$C$, there is (up to isomorphism)
a unique commutative Hopf algebra~$S(t_C)_{\Theta}$ together with a coalgebra map
\[
t: C \to S(t_C)_{\Theta}
\]
such that 
for any coalgebra map $f: C \to H'$ from~$C$ to a commutative Hopf algebra~$H'$,
there is a unique Hopf algebra map $\tilde f : S(t_C)_{\Theta} \to H'$ such that
$f = \tilde f \circ t $. The commutative Hopf algebra~$S(t_C)_{\Theta}$ is called
the \emph{free commutative Hopf algebra} generated by~$C$.

Let us give an explicit construction of~$S(t_C)_{\Theta}$.
Pick a copy~$t_C$ of the underlying vector space of~$C$
and denote the identity map from $C$ to~$t_C$ by $x\mapsto t_x$ ($x\in C$).
Let $S(t_C)$ be the symmetric algebra over the vector space~$t_C$.
By~\cite[Lemma~A.1]{AK} there is a unique linear map $x \mapsto t^{-1}_x$
from~$C$ to the field of fractions~$\Frac S(t_C)$ of~$S(t_C)$ such that for all $x\in C$,
\begin{equation*}\label{tt}
t_{x_1} \, t^{-1}_{x_{2}} = t^{-1}_{x_{1}} \,  t_{x_{2}} = \eps(x) \, 1 \, .
\end{equation*}
Then $S(t_C)_{\Theta}$ is the subalgebra of~$\Frac S(t_C)$ generated 
by all elements~$t_x$ and~$t^{-1}_x$, where $x$ runs over~$C$. 
The coproduct, counit, antipode of~$S(t_C)_{\Theta}$ are
determined for all $x\in C$ by
\begin{equation}\label{coproduct-comm-coalg}
\Delta(t_x) = t_{x_1} \otimes t_{x_2} \quad\text{and}\quad \Delta(t^{-1}_x) = t^{-1}_{x_2} \otimes t^{-1}_{x_1} \, ,
\end{equation}
\begin{equation}\label{counit-comm-coalg}
\eps(t_x) = \eps(t^{-1}_x) = \eps(x) \, ,
\end{equation}
\begin{equation}\label{ant-comm-coalg}
S(t_x) = t^{-1}_x \quad\text{and}\quad 
S(t^{-1}_x) = t_x \, .
\end{equation}
The map $t: C \to S(t_C)_{\Theta}$ is defined by $xÊ\mapsto t_x$ ($x\in C$).

By \cite[Th.\,61 and Cor.\,64]{Ta} the algebra~$S(t_C)_{\Theta}$ is a localization of~$S(t_C)$.
More precisely, if $(D_{\alpha})_{\alpha}$ is a family of finite-dimensional subcoalgebras of~$C$ 
such that $\sum_{\alpha}\, D_{\alpha}$ contains the coradical of~$C$, 
then $S(t_C)_{\Theta}$ is obtained from~$S(t_C)$ 
by inverting certain group-like elements~$\Theta_{\alpha} \in S(t_{D_{\alpha}}) \subset S(t_C)$:
\begin{equation}\label{eq-loc}
S(t_C)_{\Theta} =  S(t_C)\left [ \left(\frac{1}{\Theta_{\alpha}}\right)_{\!\alpha} \, \right] \, .
\end{equation}

In the sequel we shall need the following lemma.

\begin{lemma}\label{lem-pointed}
If the coalgebra~$C$ is pointed, then so is~$S(t_C)_{\Theta}$.
\end{lemma}

\begin{proof}
Recall the following facts from\,\cite[Lemma\,5.1.10 and Cor.\,5.3.5]{Mo}.
\begin{itemize}
\item[(i)]
If $C$ and $D$ are coalgebras with respective coradicals $C_0$ and~$D_0$, 
then the coradical $(C \otimes D)_0$ of the tensor product $C\otimes D$ of coalgebras is contained in~$C_0 \otimes D_0$.

\item[(ii)]
Given a coalgebra surjection $f: C\to D$, we have $D_0 \subset f(C_0)$.
\end{itemize}

Note that a bialgebra $B$ generated by a subcoalgebra~$C$ is pointed if $C$ is pointed. 
Indeed, by\,(i), and by\,(ii) applied to the product map $C^{\otimes n} \to C^n$,
we have
\[ 
B_0 = \left( \sum_{n\ge0} \, C^n \right)_0 = \sum_{n\ge0} \, (C^n)_0 \subset \sum_{n\geq 0}\, C_0^n \, . 
\]
Thus, $B_0$ is included in the subalgebra generated by~$C_0$, and is
generated by the group-like elements of~$C$ if $C$ is pointed.
Applying this to $B = S(t_C)$ and $C = t_C$, we deduce that
$S(t_C)$ is pointed if $C$ is pointed.

Now by\,\eqref{eq-loc} the algebra~$S(t_C)_{\Theta}$ is obtained from~$S(t_C)$ 
by inverting a (central) multiplicative subset~$T$ consisting of group-like elements. 
We consider the pointed coalgebra~$k[T^{-1}]$ spanned by the symbols $t^{-1}$, $t \in T$, 
which are supposed to be group-like.
Thus by\,(i) above, the tensor product of coalgebras $S(t_C) \otimes k[T^{-1}]$ is pointed.
We conclude by applying\,(ii) to the coalgebra surjection $S(t_C) \otimes k[T^{-1}] \to S(t_C)_{\Theta}$
given by $P\otimes t^{-1} \mapsto Pt^{-1}$ ($P\in S(t_C), t\in T$).
\end{proof}

\subsection{The $H_{\ab}$-coaction and the subalgebra~$\CC_H$}\label{subsec-coaction}

Let $H$ be a Hopf algebra and $S(t_H)_{\Theta}$ be the free commutative Hopf algebra
generated by the coalgebra underlying~$H$, as defined in the previous subsection.
We denote by $H_{\ab}$ the largest commutative Hopf algebra quotient of~$H$.
Applying the universal property of~$S(t_H)_{\Theta}$ to the canonical surjection of Hopf algebras 
$q: H \to H_{\ab}$, we obtain a unique Hopf algebra surjection 
\begin{equation}\label{qtilde}
\tilde{q} : S(t_H)_{\Theta} \to H_{\ab}
\end{equation}
such that $\tilde{q}(t_x) = q(x)$ and $\tilde{q}(t^{-1}_x) = S(q(x)) = q(S(x))$ for all $x\in H$.
Using~$\tilde{q}$, we may equip $S(t_H)_{\Theta}$ with a right $H_{\ab}$-comodule algebra structure:
its coaction is the algebra map~$\delta_S$ defined as the following composition:
\begin{equation}\label{coaction-SH}
\delta_S : S(t_H)_{\Theta} \overset{\Delta}{\longrightarrow} S(t_H)_{\Theta} \otimes S(t_H)_{\Theta} 
\overset{\id\otimes \tilde{q}}{\longrightarrow} S(t_H)_{\Theta} \otimes H_{\ab} \, .
\end{equation}
We have $\delta_S(t_x) = t_{x_1} \otimes q(x_2)$ for all $x\in H$.

The previous formula shows that $\delta_S$ sends $S(t_H)$ to $S(t_H)\otimes H_{\ab}$, 
which implies that $S(t_H)$ is an $H_{\ab}$-comodule subalgebra of~$S(t_H)_{\Theta}$. 

Let $\CC_H$ be the subalgebra of right $H_{\ab}$-coinvariants of~$S(t_H)_{\Theta}$:
\begin{equation*}\label{def-CH}
\CC_H = S(t_H)_{\Theta}^{\co-H_{\ab}} = \{ a \in S(t_H)_{\Theta} \, | \, \delta_S(a)  = a \otimes 1 \} \, .
\end{equation*}
It contains the subalgebra of right $H_{\ab}$-coinvariants of~$S(t_H)$:
\[
S(t_H)^{\co-H_{\ab}} \subset \CC_H = S(t_H)_{\Theta}^{\co-H_{\ab}} \, .
\]
Since $S(t_H)_{\Theta}$ is a localization of~$S(t_H)$,
we may wonder whether
$\CC_H = S(t_H)_{\Theta}^{\co-H_{\ab}}$ is likewise a localization of~$S(t_H)^{\co-H_{\ab}}$.
The following provides an answer.

\begin{prop}\label{prop-CH-loc}
The algebra $\CC_H$ is a localization of~$S(t_H)^{\co-H_{\ab}}$.
\end{prop}

For the proof we need the following lemma.

\begin{lemma}\label{lem-CH-loc}
Each element $\Theta_{\alpha}$ as in\,\eqref{eq-loc} may be chosen so that $\tilde{q}(\Theta_{\alpha})=1$.
\end{lemma}

\begin{proof}
According to\,\cite[Sect.\,11]{Ta}, $\Theta_{\alpha}$ is given by the determinant of
a square matrix $\big( t_{x_{i,j}} \big)$, where $x_{i,j} \in D_{\alpha}$ and
\begin{equation*}\label{comatric}
\Delta(x_{i,j}) = \sum_{\ell} \, x_{i,\ell} \otimes x_{\ell, j} \quad\text{and}\quad \varepsilon(x_{i,j})=\delta_{i,j} 
\end{equation*}
for all $i,j$. 
These formulas imply that in~$S(t_H)$ the matrix $M = (t_{x_{i,j}})$ satisfies the equations
$\Delta(M) = (M \otimes I)(I\otimes M)$ and $\varepsilon(M)=I$,
where $I$ is an identity matrix. This implies that $\Theta_{\alpha} = \det M$ is group-like, as is well known.

Let us now consider the image~$S(x_{i,j})$ of~$x_{i,j}$ under the antipode of~$H$.  
The transpose~$M'$ of the matrix~$( t_{S(x_{i,j})})$ satisfies the same formulas as above.
Therefore, $\Theta'_{\alpha} = \det M'$ is a group-like element of~$S(t_{S(D_{\alpha})})$. 
We may replace $D_{\alpha}$ by $D_{\alpha} + S(D_{\alpha})$ and
$\Theta_{\alpha}$ by $\Theta_{\alpha} \Theta'_{\alpha}$. 
It remains to prove that $\tilde{q}(\Theta_{\alpha} \Theta'_{\alpha}) = 1$. 
But this follows from the fact that the matrices
\[ 
\left( q(x_{i,j}) \right) = \left( \tilde{q}(t_{x_{i,j}}) \right) 
\quad \text{and}\quad
\left(q(S(x_{i,j})) \right) = \left( \tilde{q}(t_{S (x_{i,j})}) \right)
\]
with entries in $H_{\ab}$ are inverses of each other. 
\end{proof}

\begin{proof}[Proof of Proposition\,\ref{prop-CH-loc}]
Choose $(\Theta_{\alpha})_{\alpha}$ as in Lemma\,\ref{lem-CH-loc}. 
Let $\bar{P}$ be an element of $\CC_H = S(t_H)_{\Theta}^{\co-H_{\ab}}$. 
Then $\bar{P}\, \Theta \in S(t_H)$ for some finite product $\Theta$ of the elements~$\Theta_{\alpha}$;
the element~$\Theta$ is group-like and $\tilde{q}(\Theta) = 1$.
Define $P= \bar{P}\, \Theta$. 
Since $\delta_S(\bar{P}) = \bar{P} \otimes 1$ and
$\delta_S(\Theta) = \Theta \otimes \tilde{q}(\Theta) = \Theta \otimes 1$,
we have $\delta_S(P)=P \otimes 1$, which means that $P$ belongs to~$S(t_H)^{\co-H_{\ab}}$.
\end{proof}

\subsection{The generic base algebra}\label{subsec-generic}

Under some conditions on~$H$, the commutative algebra $\CC_H$ has an alternative description,
which we now present.

For any pair $(x,y)$ of elements of a Hopf algebra~$H$ consider the following elements of~$S(t_H)_{\Theta}$:
\begin{equation}\label{sigma-def}
\sigma(x,y) = t_{x_1} \, t_{y_1} \, t^{-1}_{x_2 y_2} 
\quad\text{and}\quad
\sigma^{-1}(x,y) = t_{x_1 y_1} \, t^{-1}_{x_2}  \, t^{-1}_{y_2}\, .
\end{equation} 

Following~\cite[Sect.~5]{AK} and~\cite[Sect.~3]{Ka1},
we define the \emph{generic base algebra}~$\BB_H$ attached to the Hopf algebra~$H$
to be the subalgebra of~$S(t_H)_{\Theta}$ 
generated by all elements~$\sigma(x,y)$ and~$\sigma^{-1}(x,y)$.
Since $\BB_H$ sits inside~$S(t_H)_{\Theta}$, it is a domain.

The generic base algebra~$\BB_H$ is the algebra of coinvariants of a right $H$-comodule algebra~$\AA_H$
parametrizing the ``forms'' of~$H$ (for details, see~\cite{AK, Ka1} and also Section\,\ref{subsec-pb2} of this paper).

If~$H$ is \emph{finite-dimensional}, then by Theorem\,3.6 and Corollary\,3.7 of~\cite{KM}
the following holds:

\begin{itemize}
\item[(a)]
$\BB_H$ is a finitely generated smooth Noetherian domain of Krull dimension equal to~$\dim\,  H$;

\item[(b)]
the embedding $\BB_H \subset S(t_H)_{\Theta}$ turns $S(t_H)_{\Theta}$ into a finitely generated projective $\BB_H$-module.
\end{itemize}

We now relate $\BB_H$ to the algebra~$\CC_H$ of $H_{\ab}$-coinvariants introduced in
the previous subsection.

\begin{prop}\label{B=C} 
Let $H$ be a Hopf algebra.

(a) We have $\BB_H \subset \CC_H$.

(b) The equality $\BB_H = \CC_H$ holds if one of the following conditions is satisfied:
\begin{itemize}
\item[(i)] $H$ is finite-dimensional;

\item[(ii)] $H$ is cocommutative;

\item[(iii)] $H$ is pointed and each element of the kernel of the natural homomorphism
$G(H)_{\ab} \to G(H_{\ab})$ is of finite order;

\item[(iv)] $H$ is commutative.
\end{itemize}
\end{prop}

Here $G(H)_{\ab}$ denote the abelianization of the group~$G(H)$ of group-like elements of~$H$.
Note that in view of\,(i), Conditions\,(ii)--(iv) above are only relevant for infinite-dimensional Hopf algebras.

\begin{proof}
(a) It suffices to show that all elements $\sigma(x,y)$ and $\sigma^{-1}(x,y)$ are $H_{\ab}$-coin\-var\-iants.
We shall check this for~$\sigma(x,y)$, the proof for~$\sigma^{-1}(x,y)$ being similar.
By\,\cite[Lemma\,3.3]{KM},
\begin{equation*}
\delta_S\left( \sigma(x,y) \right) = t_{x_1} \, t_{y_1} \, t^{-1}_{x_3 y_3}  \otimes \tilde{q} \left( \sigma(x_2,y_2) \right) .
\end{equation*}
Now, 
\begin{eqnarray*}
\tilde{q} \left( \sigma(x,y) \right) & = & \tilde{q}( t_{x_1} ) \, \tilde{q}( t_{y_1} ) \, \tilde{q}( t^{-1}_{x_2 y_2}) 
= \tilde{q}( t_{x_1} ) \, \tilde{q}( t_{y_1} ) \, \tilde{q}( S(t_{x_2 y_2})) \\
& = & q(x_1) \, q(y_1)  \, q(S(x_2 y_2)) 
= q \left( (x_1y_1)\, S(x_2 y_2) \right)\\
& = & \eps(xy) \, q(1) \, .
\end{eqnarray*}
Consequently, 
\begin{equation*}
\delta_S\left( \sigma(x,y) \right) = t_{x_1} \, t_{y_1} \, \eps(x_2 y_2) \, t^{-1}_{x_3 y_3}  \otimes q(1) 
= \sigma(x,y) \otimes q(1)  \, ,
\end{equation*}
which shows that $\sigma(x,y)$ belongs to~$\CC_H$.

(b) (i)--(iii) By Theorems\,3.6, \,3.8, and\,3.9 of\,\cite{KM}, $S(t_H)_{\Theta}$ is faithfully flat as a $\BB_H$-module. 
Then by\,\cite[Th.\,3]{Ta2}, $\BB_H = S(t_H)_{\Theta}^{\co-Q}$, where $Q$ is the quotient Hopf algebra of~$S(t_H)_{\Theta}$
by the ideal~$(\BB_H^+)$ generated by $\BB_H^+ = \BB_H \cap \ker\eps$.
By\,\cite[Prop.\,3.1]{KM}, there is a natural identification $Q\cong H_{\ab}$, which allows us to conclude.

(iv) As is shown in Part\,(a) of the proof of~\cite[Th.\,3.13]{KM},
the mapping $c \otimes x \mapsto c t_x$ gives an 
isomorphism $\CC_H \otimes H \cong S(t_H)_{\Theta}$.
Moreover,  since the natural inclusion
$\BB_H\otimes H \subset \CC_H \otimes H$, composed with the isomorphism, gives
a surjection by~\cite[Lemma\,3.2]{KM}, we have $\BB_H =\CC_H$.
\end{proof}

\section{\large The generalized Noether problem}\label{sec-Noether-pb}

We now state the problem which gives the title to the paper.
We provide answers for certain classes of Hopf algebras.

\subsection{The problem and its relationship with the classical Noether problem}\label{subsec-GNP}

\begin{Pb}[GNP]
\emph{Given a field~$k$ and a finite-dimensional Hopf algebra~$H$ over~$k$, 
is the algebra~$\BB_H$ a localization of a polynomial algebra in finitely many variables?}
\end{Pb}

We call Problem (GNP) the \emph{generalized Noether problem}. 
Let us first show how (GNP) is related to the classical Noether problem for finite groups.

Let $G$ be a finite group and $H = O_k(G)$ the dual Hopf algebra of the group algebra~$k[G]$.
The elements of~$H$ can be seen as $k$-valued fonctions on~$G$;
denote by~$e_g$ the function which vanishes everywhere on~$G$, except at the element~$g$, where $e_g(g) = 1$.
On the basis $\{ e_g\}_{g\in G}$ the coproduct, the counit and the antipode of~$H$ are given by
\begin{equation}\label{coprod-OG}
\Delta(e_g) = \sum_{h \in G} \, e_{gh^{-1}} \otimes e_h \, ,
\quad
S(e_g) = e_{g^{-1}} \, ,
\end{equation}
$\eps(e_g) = 1$ if $g=1$ is the identity element of~$G$, and $\eps(e_g) = 0$ otherwise.

\begin{prop}\label{GNP-NP}
If (GNP) has a positive answer for the Hopf algebra~$H = O_k(G)$, 
then so has the classical Noether problem for~$G$ and~$k$. 
\end{prop}

In this sense, (GNP) is an extension of Noether's problem.

\begin{proof}
If we set $t_g = t_{e_{g^{-1}}}$ for any $g\in G$,
then $S(t_H) = k[\, t_g \, |\, g\in G \, ]$ is the bialgebra with coproduct given for $g\in G$ by
\begin{equation}\label{coprod-SOG}
\Delta(t_g) = \sum_{h \in G} \, t_{hg} \otimes t_{h^{-1}} \, .
\end{equation}
By~\cite[Ex.~B.5]{AK}, the algebra $S(t_H)_{\Theta}$ is the following localization of~$S(t_H)$:
\[
S(t_H)_{\Theta} = k[\, t_g \, |\, g\in G \, ] \left[ \frac{1}{\Theta_G} \right] \, ,
\]
where $\Theta_G = \det(t_{gh^{-1}})_{g,h \in G}$ is {Dedekind's group determinant}.

Under the algebra map $\tilde{q}: S(t_H)_{\Theta} \to H_{\ab} = H$ of\,\eqref{qtilde},
the algebra~$S(t_H)_{\Theta}$ becomes a right $H$-comodule algebra
with coaction\,\eqref{coaction-SH}.
Now it is well known that any right $H$-comodule algebra structure on an algebra~$A$ is the same as
a left action $(h,a) \mapsto h\cdot a$ of~$G$ on~$A$ by algebra automorphisms:
the $H$-coaction~$\delta$ on~$A$ is related to the $G$-action by
\begin{equation}\label{coaction-action}
\delta(a) = \sum_{h\in G} \, h\cdot a \otimes e_h \, .
\end{equation}
Moreover, the subalgebra of coinvariants~$A^{\co-H}$ coincides with the subalgebra~$A^G$
of $G$-invariant elements of~$A$.
Since for $H = O_k(G)$ 
we have $\tilde{q}(t_g) = e_{g^{-1}}$ for all $g\in G$, 
it follows from\,\eqref{coaction-SH} and from\,\eqref{coprod-SOG}
that
\[
\delta_S(t_g) = \sum_{h\in G} \, t_{hg} \otimes e_h \, .
\]
Comparing this formula with\,\eqref{coaction-action}, we see that
the corresponding left action of~$G$ on~$S(t_H)_{\Theta}$ is given by
$h \cdot t_g = t_{hg}$, which is precisely the one presented in the introduction.
One easily checks that the square~$\Theta_G^2$ of the Dedekind determinant is $G$-invariant, 
so that by Proposition\,\ref{B=C},
\[
\BB_H = \CC_H = \bigl(S(t_H)_{\Theta}\bigr)^G 
= k[\, t_g \, |\, g\in G \, ]{\,}^G \left[ \frac{1}{\Theta_G^2} \right] \, .
\]

If $\BB_H = k[\, t_g \, |\, g\in G \, ]{\,}^G [1/\Theta_G^2]$ is the localization of a polynomial algebra in finitely many variables,
then its fraction field is a purely transcendental extension of~$k$. 
This fraction field is also the fraction field of $k[\, t_g \, |\, g\in G \, ]{\,}^G$.
Now the latter is $k(\, t_g \, |\, g\in G \, ){\,}^G$: 
indeed, if $F = P/Q$ is a $G$-invariant fraction with polynomial numerator and
denominator, then it can be written as $F = P'/Q'$, 
where $P' = P \, \prod_{g\neq 1}\, (g\cdot Q)$ and $Q' = \prod_{g\in G}\, (g\cdot Q)$ are invariant polynomials.
Therefore, the field $k(\, t_g \, |\, g\in G \, ){\,}^G$ is a purely transcendental extension of~$k$. 
\end{proof}

\subsection{Positive answers to (GNP)}\label{subsec-positive}

Before stating the main result of this section, we present a class of Hopf algebras 
for which it is easy to provide a positive answer to~(GNP).

\begin{exa}\label{ex-kG}
Let $H= k[G]$ be the Hopf algebra of a finite group~$G$. 
It is easy to see that the maximal commutative Hopf algebra quotient is
the group algebra $H_{\ab} = k[G_{\ab}]$, where $G_{\ab}$ is the abelianization of~$G$. 
By Proposition\,\ref{B=C}, $\BB_H = \CC_H$. It is the group algebra of a finitely generated free abelian group;
more precisely, by\,\cite[Prop.\,9 and\,14]{AHN} (see also\,\cite[Prop.\,A.1]{IK}),
\begin{equation*}\label{kY}
\BB_{k[G]} = \CC_{k[G]} = S(t_{k[G]})_{\Theta}^{\co-k[G_{\ab}]} = k[Y_G] \, ,
\end{equation*}
where $Y_G$ is the kernel of the homomorphism $\ZZ^G \to G_{\ab}\, ; t_g \mapsto \bar{g}$ $(g\in G)$.
(Here $\ZZ^G$ is the free abelian group generated by the symbols~$t_g$ ($g\in G$)
and $\bar{g}$ denotes the image of~$g\in G$ in~$G_{\ab}$.)
Since $Y_G$ is a finite index subgroup of~$\ZZ^G$, it is a free abelian group of the same rank~$\card\ G$.
(A basis of~$Y_G$ is described in Lemma\,\ref{lem-gen1}).
Therefore $\BB_{k[G]} = k[Y_G] $ is an algebra of Laurent polynomials in finitely many ($\card\ G$) variables,
which shows that (GNP) has a positive answer for $H = k[G]$.
\end{exa}

The main result of this section is the following.

\begin{theorem}\label{th-GNP}
Let $H$ be a finite-dimensional pointed Hopf algebra such that $H_{\ab}$ is a group algebra.
Then (GNP) has a positive answer for~$H$.
\end{theorem}

Recall from Lemma\,\ref{lem:com} that the hypothesis on~$H_{\ab}$ is satisfied when
the base field~$k$ is of characteristic zero

The theorem had previously been established for special classes of finite-dimen\-sional pointed Hopf algebras:
for finite group algebras by Aljadeff, Haile and Natapov\,\cite{AHN}
as detailed in Example\,\ref{ex-kG};
furthermore, for the Taft algebras, for the Hopf algebras~$E(n)$ and for certain monomial Hopf algebras
by Iyer and the first-named author (see\,\cite[Th.\,2.1, Th.\,3.1, Th.\,4.1]{IK}). In all these cases,
$H_{\ab}$ is a group algebra.

\begin{proof}
We set $G = G(H)$. Let $R = H/H(k[G])^+$ and choose $\gamma: H \to k[G]$ as in Lemma\,\ref{lemma:retraction}. 
Define $\lambda : R \to k[G] \otimes R$ as in\,\eqref{lambda}, 
and construct thereby the smash-coproduct coalgebra $R\rbtimess k[G]$. 
Identify $H$ with $R\rbtimess k[G]$ via the $k[G]$-module coalgebra isomorphism~$\xi$ given by\,\eqref{xi}. 
Since $H = R^+G \oplus k[G]$, we have
\begin{equation}\label{tensor_product}
S(t_H)_{\Theta} = S(t_{R^+G}) \otimes S(t_{k[G]})_{\Theta}.
\end{equation}
Note that this is an identification of $S(t_{k[G]})_{\Theta}$-algebras.  

Consider the Hopf algebra surjections
\[ 
\tilde{\gamma} : S(t_H)_{\Theta} \to S(t_{k[G]})_{\Theta}
\quad \text{and}\quad
\tilde{q} : S(t_H)_{\Theta} \to H_{\ab} 
\]
induced from the coalgebra maps~$\gamma$ and~$q$, respectively.
Regard $S(t_H)_{\Theta}$ as a right $S(t_{k[G]})_{\Theta}$-comodule algebra along~$\tilde{\gamma}$, 
and let $A$ denote the (left coideal) subalgebra of~$S(t_H)_{\Theta}$ consisting of all right $S(t_{k[G]})_{\Theta}$-coinvariants. 
Since $\tilde{\gamma}$ is a Hopf algebra retraction of the inclusion $S(t_{k[G]})_{\Theta}\subset S(t_H)_{\Theta}$,
the product map gives a natural isomorphism
\begin{equation}\label{Aisom}
A \otimes S(t_{k[G]})_{\Theta} = S(t_H)_{\Theta} 
\end{equation}
of right $S(t_{k[G]})_{\Theta}$-comodules and of $S(t_{k[G]})_{\Theta}$-algebras. 
Recall from Lemma\,\ref{lemma:retraction} that $q = q|_{k[G]}\circ \gamma$. It follows that
\begin{equation}\label{eq-factor-qtilde}
\tilde{q} = \tilde{q}|_{S(t_{k[G]})_{\Theta}}\circ \tilde{\gamma} \, .
\end{equation}
Regard $S(t_{k[G]})_{\Theta}$ as a right $H_{\ab}$-comodule algebra along $\tilde{q}|_{S(t_{k[G]})_{\Theta}}$. 
Then since $A$ is right $\tilde{q}$-coinvariant, it follows from\,\eqref{eq-factor-qtilde}
that the isomorphism given in\,\eqref{Aisom} restricts to an isomorphism
\begin{equation}\label{Aisom2}
A\otimes S(t_{k[G]})_{\Theta}^{\co-H_{\ab}} = S(t_H)_{\Theta}^{\co-H_{\ab}} = \mathcal{B}_H \, . 
\end{equation}
Now $S(t_{k[G]})_{\Theta}^{\co-H_{\ab}}$, being the algebra of coinvariants along a Hopf algebra map, is
a Hopf subalgebra of~$S(t_{k[G]})_{\Theta}$; 
the latter being the group algebra of a finitely generated free abelian group
(see Example\,\ref{ex-kG}), so is $S(t_{k[G]})_{\Theta}^{\co-H_{\ab}}$;
in other words, $S(t_{k[G]})_{\Theta}^{\co-H_{\ab}} \cong k[\ZZ^{\ell}]$ for some non-negative integer~$\ell$.
We remark that $\ell = \card\ G$, since the short exact sequence 
$S(t_{k[G]})_{\Theta}^{\co-H_{\ab}} \to S(t_{k[G]})_{\Theta} \to H_{\ab}$ of commutative Hopf algebras 
restricts on the level of group-like elements to the short exact sequence 
\begin{equation}\label{eq-ses}
0 \to \ZZ^{\ell} \to \ZZ^{G} \to G(H_{\ab}) \to 0
\end{equation}
of abelian groups with finite cokernel. 

It remains to prove that $A$ is a polynomial algebra with finitely many variables. 
But this follows since one sees from\,\eqref{tensor_product} and\,\eqref{Aisom} that
\begin{equation}\label{S_isom_to_A}  
S(t_{R^+G}) \cong S(t_H)_{\Theta}/(S(t_{k[G]})_{\Theta}^+) \cong A \, .
\end{equation}
In conclusion, the isomorphism $\BB_H \cong S(t_{R^+G}) \otimes k[\ZZ^{\ell}]$ provides a positive
answer to~(GNP).
We remark that, if we set $n = \dim\ H$, then 
$S(t_{R^+G}) \cong k[\NN^{n-\ell}]$, since we see from $H = R^G \oplus k[G]$ that 
$\dim\ R^+G = n - \ell$. 
\end{proof}

In the previous proof we have actually established the more precise isomorphism
\[
\BB_H \cong k[\ZZ^{\ell}] \otimes k [\NN^{n- \ell}] \, , \quad n = \dim\ H \;\text{and}\; \ell = \card\ G(H) \, .
\]
We will come back to it in Section\,\ref{subsec-bounding}.

\begin{rem}\emph{(A variant of~(GHP))}
By Proposition\,\ref{B=C}, $\BB_H = S(t_H)_{\Theta}^{\co-H_{\ab}}$ if $H$ is finite-dimen\-sional.
As was seen in Section\,\ref{subsec-coaction},
$S(t_H)$ is an $H_{\ab}$-comodule subalgebra of~$S(t_H)_{\Theta}$,
so that we can consider the coinvariant subalgebra~$S(t_H)^{\co-H_{\ab}}$.
A~variant of~(GNP) is the following: 
is~$S(t_H)^{\co-H_{\ab}}$ a polynomial algebra in finitely many variables?
The following example shows that this variant may have a negative answer at the same time as
(GNP) has a positive one.

Let $H$ be the four-dimensional Sweedler algebra. 
Proceeding as in the proof of~\cite[Th.\,2.1]{IK}, it is easy to check that $S(t_H)^{\co-H_{\ab}}$
is spanned by the monomials $t_1^a t_x^b t_y^c t_z^d$ ($a,b,c,d\in \NN$) such that $b+c$ is even.
Hence, $S(t_H)^{\co-H_{\ab}}$ is the subalgebra of~$k[t_1, t_x, t_y, t_z]$ 
generated by $t_1$, $t_x^2$, $t_x t_y$, $t_y^2$, $t_z$;
this is not a polynomial algebra, whereas $\BB_H= S(t_H)_{\Theta}^{\co-H_{\ab}}$, 
which is obtained from the previous algebra by inverting $t_1$ and $t_x^2$, 
is a localization of the polynomial algebra $k[t_1, t_x^2, t_xt_y, t_z]$.
\end{rem}

\subsection{Bounding the degrees of generators}\label{subsec-bounding}

Given a finite group~$G$, 
consider the polynomial algebra $S(t_G) = k[t_g \, |\, g \in G]$ in the variables~$t_g$ ($g\in G$)
and let $G$ act on the variables~$t_g$, hence on~$S(t_G)$, as in the introduction. 
Let $S(t_G)^G$ be the subalgebra of~$S(t_G)$ of $G$-invariant polynomials.
We denote by~$\beta(G)$ the smallest integer~$\beta$ such that $S(t_G)^G$
is generated by homogeneous polynomials of degree~$\leq \beta$.

In~\cite{No0} Emmy Noether proved that $\beta(G) \leq \card\, G$.
It is easy to check that, if the group~$G$ is cyclic, then $\beta(G) =  \card\, G$. 
In her thesis~\cite{Sch}, Barbara Schmid proved a conjecture by Kraft, 
namely $\beta(G) < \card\, G$ if $G$ is not cyclic. 

We now prove a similar result for a finite-dimensional pointed Hopf algebra~$H$ such that $H_{\ab}$ is 
a group algebra. Set $\bar{G} = G(H_{\ab})$; equivalently, $H_{\ab} = k[\bar{G}]$.

Define two non-negative integers~$d,r$ related to the finite abelian group~$\bar{G}$ as follows: 
if $\bar{G}$ is trivial, set $d  = r = 0$; if $\bar{G}$ is not trivial, let
\begin{equation}\label{prime-decomp}
\bar{G} = \ZZ/p_1^{e_1} \times \dots  \times \ZZ/p_r^{e_r} \, ,\quad r \geq 1 , \quad  p_i\ \text{primes}, \quad e_i  \geq 1 
\end{equation}
be the primary decomposition of~$\bar{G}$, and set $d = p_1^{e_1} + \cdots + p_r^{e_r}$ ($\geq 2$).

\begin{theorem}\label{th-GNP-degrees}
Under the hypotheses of Theorem\,\ref{th-GNP} and with the above notation, we have
\[
\BB_H = k[u_1^{\pm 1}, \ldots, u_{\ell}^{\pm 1}, u_{\ell+1}, \ldots, u_n]  \, ,
\]
where $n = \dim\ H$ and $\ell = \card\ G(H)$
and where $u_1, \ldots, u_{\ell}$ are monomials in the variables~$t_g$ of degree~$\leq d-r+1$
and $u_{\ell+1}, \ldots, u_n$ are monomials of degree $\leq 2$.
\end{theorem}

Note that $d-r+1\leq \card\, \bar{G}$, but $d-r+1$ may be much smaller:
for instance, if $\bar{G}$ is a $p$-group of order~$p^e$, 
then $d-r+1$ ranges from~$p^e$ when $\bar{G} \cong \ZZ/p^e$ is cyclic
down to~$e(p-1) +1$ when $\bar{G} \cong (\ZZ/p)^e$ is elementary abelian. 

For the proof of Theorem\,\ref{th-GNP-degrees} we shall make use of the fact
that $S(t_{k[G]})_{\Theta}^{\co-H_{\ab}}$ is a subalgebra of~$\BB_H$ inside~$S(t_{k[G]})_{\Theta}$
(see proof of Theorem\,\ref{th-GNP}) and of the two lemmas below.

To state the first lemma we need the following notation. 
By Item\,(ii) in the proof of Lemma\,\ref{lem-pointed}, the surjection of pointed Hopf algebras 
$H \to H_{\ab}$ restricts to a surjection of groups $G \to \bar{G}$.
Suppose that the group $\bar{G}$ is non-trivial.
For $1 \leq i \leq r$ let $s_i$ be a generator of the summand~$\ZZ/p_i^{e_i}$ 
in the primary decomposition\,\eqref{prime-decomp} of~$\bar{G}$,
and fix an element $\sigma_i \in G$ whose image in~$\bar{G}$ is~$s_i$.
Given an element $g\in G$, its image in~$\bar{G}$ is of the form $s_1^{f_1(g)} \cdots s_r^{f_r(g)}$
($0 \leq f_i(g) < p_i^{e_i}$, $1 \leq i \leq r$).
For $g\in G$ different from the identity element~$e$ and of the elements $\sigma_1, \ldots, \sigma_r$, 
let $I  = \{ i \in  \{1, \ldots, r\} \, |\, f_i(g)  \neq 0Ê\}$ and 
set
\[
u_g = t_g \prod_{i \in I} \, t_{\sigma_i}^{p_i^{e_i} - f_i(g)}
\in S(t_{k[G]}) \subset S(t_H) \, .
\]

\begin{lemma}\label{lem-gen1}
The set consisting of the monomials (in the $t$-variables)
$t_e, t_{\sigma_1}^{p_1^{e_1}} , \ldots, t_{\sigma_r}^{p_r^{e_r}}$ and of the above monomials~$u_g$ 
is a (multiplicatively written) basis of the kernel~$\ZZ^{\ell}$ 
appearing in the short exact sequence of groups\,\eqref{eq-ses}.
\end{lemma}

\begin{proof}
These elements clearly belong to the kernel~$\ZZ^{\ell}$.
Their matrix with respect to the suitably ordered basis~$(t_g)_{g\in G}$ of~$\ZZ^{G}$
is triangular with determinant equal to $p_1^{e_1} \cdots p_r^{e_r} = \bar{G}$. 
Therefore, they form a basis of~$\ZZ^{\ell}$.
\end{proof}

\begin{lemma}\label{lem-gen2} 
As an $S(t_{k[G]})_{\Theta}^{\co-H_{\ab}}$-algebra,  
$\BB_H$ is freely generated by $n-\ell$ monomials of degree $\leq 2$ in~$S(t_H)$.
\end{lemma}

\begin{proof}
We return to the situation of the proof of Theorem~\ref{th-GNP}:
we have the identification $H = R \rbtimess k[G]$ of right $k[G]$-module coalgebras, 
under which the right $k[G]$-module coalgebra map
$\gamma : H \to k[G]$ is identified with $\varepsilon \otimes \mathrm{id} : R \otimes k[G] \to k[G]$.

We claim that $R = R\otimes k$ precisely consists of those elements of~$H$
which are right $k[G]$-coinvariant along~$\gamma$, that is,
\[ 
R = \left\{ h \in H \, | \, (\id \otimes \gamma )\circ \Delta(h) = h \otimes 1 \right\} \, . 
\]
Indeed, this holds for smash coproducts in general. 
In our situation, recall from the proof of Lemma\,\ref{lemma:retraction} 
that $R$ is a left $k[G]$-comodule coalgebra with respect to $\lambda : R \to k[G]\otimes R$. 
For $x \in R$ and $g \in G$, the coproduct $\Delta(x \otimes g)$ on $H$ is given by
\[ 
\Delta(x \otimes g)= (x_{(1)}\otimes (x_{(2)})^{(-1)}g) \otimes ((x_{(2)})^{(0)}\otimes g) \, ,
\]
where $\lambda(x) = x^{(-1)}\otimes x^{(0)}$. 
Since $(\id \otimes \varepsilon) \circ \lambda (x) = \varepsilon(x)1$, it follows that
\[
(\id \otimes \gamma)\circ \Delta (x\otimes g) = (x \otimes g) \otimes g \, . 
\]
Thus the right $k[G]$-comodule structure on $H = R \otimes k[G]$ is the natural one given by
the right tensor factor~$k[G]$. This proves the claim.

Now choose arbitrarily a $k$-basis $\{ x_i \}_i$ of~$R^+$. 
Then we have the $k$-basis $\{ x_i g \}_{i,g}$ of~$R^+G$, where $g$ runs over~$G$. 
Note $\dim\ R^+G=n-\ell$. 
One sees from the claim above that the coproduct $\Delta(t_{x_i g})$ on~$S(t_H)_{\Theta}$, 
composed with $\id \otimes \widetilde{\gamma}$, turns into
\[
(\id \otimes \widetilde{\gamma})\circ \Delta(t_{x_ig})= t_{x_ig}\otimes t_g \, .
\]
Obviously, $\{ t_{x_ig} \}_{i,g}$ is a set of free generators of~$S(t_{R^+G})$. 
Each $t_{x_ig}$ is congruent, modulo~$(S(t_{k[G]})_{\Theta}^+)$, to $t_{x_ig}(t_g)^{-1}$, 
which is seen, by the last equation, to be in~$A$. 
Recall from \eqref{S_isom_to_A} and \eqref{Aisom2} the isomorphisms
\[
S(t_{R^+G}) \cong S(t_H)_{\Theta}/(S(t_{k[G]})_{\Theta}^+) \cong A\, ,\qquad 
A \otimes S(t_{k[G]})_{\Theta}^{\co-H_{\ab}} \cong \BB_H. 
\]
It follows that $\{ t_{x_ig}(t_g)^{-1} \}_{i,g}$ is a set of free generators of the $k$-algebra~$A$, 
and is a set of free generators of the $S(t_{k[G]})_{\Theta}^{\co-H_{\ab}}$-algebra~$\BB_H$. 
Multiplying $t_{x_ig}(t_g)^{-1}$ by the unit $t_g t_{g^{-1}}$ in $S(t_{k[G]})_{\Theta}^{\co-H_{\ab}}$,
we obtain the set $\{ t_{x_ig}t_{g^{-1}} \}_{i,g}$ of free generators of the 
$S(t_{k[G]})_{\Theta}^{\co-H_{\ab}}$-algebra~$\BB_H$. 
The set consists of monomials of degree~$2$ in~$S(t_H)$, among which $t_{x_i}t_e$ may be replaced 
by the degree-one monomial~$t_{x_i}$ (since $t_e$ is a unit in~$S(t_{k[G]})_{\Theta}^{\co-H_{\ab}}$). 
Thus the lemma follows. 
\end{proof}

\begin{proof}[Proof of Theorem\,\ref{th-GNP-degrees}]
It follows from Lemma\,\ref{lem-gen2} that inside~$S(t_{k[G]})_{\Theta}$
we have
$\BB_H = S(t_{k[G]})_{\Theta}^{\co-H_{\ab}}[u_{\ell+1}, \ldots, u_n]$,
where $u_{\ell+1}, \ldots, u_n$ are monomials of~$S(t_H)$ of degree~$\leq 2$.

Now by the proof of Theorem\,\ref{th-GNP} the subalgebra~$S(t_{k[G]})_{\Theta}^{\co-H_{\ab}}$
is the algebra of the kernel~$\ZZ^{\ell}$ of the surjection $\ZZ^G \to \bar{G} = G(H_{\ab})$
appearing in the short exact sequence of groups\,\eqref{eq-ses}.

If $\bar{G}$ is trivial, then $\ZZ^{\ell} = \ZZ^{G}$, 
hence $S(t_{k[G]})_{\Theta}^{\co-H_{\ab}} = S(t_{k[G]})_{\Theta}$, 
which is the algebra of Laurent polynomials in the monomials~$t_g$ ($g\in G$),
which are of degree $1 = d - r +1$. 

Otherwise, $S(t_{k[G]})_{\Theta}^{\co-H_{\ab}}$ is the algebra of Laurent polynomials
in the basis elements described in Lemma\,\ref{lem-gen1}. To complete the proof, it is enough
to bound the degree of these elements. 
Now each monomial $u_g $ is of degree
$1 + \sum_{i\in I} \, (p_i^{e_i} - f_i(g))$,
which is smaller than or equal to 
\[
1 + \sum_{i\in I} \, (p_i^{e_i} - 1) \leq 1 + \sum_{i= 1}^r \, (p_i^{e_i} - 1) = d-r+1 \, .
\]
We conclude by observing that the degrees of the remaining monomials
$t_e, t_{\sigma_1}^{p_1^{e_1}} , \ldots$, $t_{\sigma_r}^{p_r^{e_r}}$
are not larger than $d-r+1$.
\end{proof}

\begin{rem}
Let $H$ be a pointed Hopf algebra such that $H_{\ab} = k[\bar{G}]$ is a group algebra. 
As observed above, the canonical surjection of pointed Hopf algebras $H \to H_{\ab}$ restricts to a
surjection $k[G] \to k[\bar{G}]$, where $G = G(H)$. Denoting the abelianization of~$G$ by~$G_{\ab}$,
we see that the surjection $k[G] \to k[\bar{G}]$ in turn induces 
a surjection $\varphi: k[G_{\ab}] \to H_{\ab} = k[\bar{G}]$. 

The map~$\varphi$ is an isomorphism if the Hopf algebra embedding $k[G] \subset H$ splits as a Hopf algebra map.
Indeed, in this case, a splitting $H \to k[G]$ composed with the canonical map $k[G] \to k[G_{\ab}]$
induces a Hopf algebra map $H_{\ab} =k[\bar{G}] \to k[G_{\ab}]$. 
The latter map is an inverse of~$\varphi$ since both maps are induced from the identity on~$k[G]$.

In general $\varphi: k[G_{\ab}] \to H_{\ab}$ is not an isomorphism, as is seen from the following examples.
Consider the pointed Hopf algebra $H=U_q(\mathfrak{sl}_2)$ as given in\,\cite[p.\,122]{Ka}. 
The group~$G= G(H)$ is infinite cyclic, hence $G_{\ab} \cong \ZZ$.
Now the relations (1.10)--(1.12) in \emph{loc. cit.} imply that $E=0$, $F=0$ and $K=K^{-1}$ in~$H_{\ab}$. 
Therefore, $H_{\ab} \cong k[\ZZ/2]$.

Similarly let $H=\overline{U}_q$ be the finite-dimensional Hopf algebra defined 
in\,\cite[p.\,136]{Ka}
as the quotient of~$U_q(\mathfrak{sl}_2)$ by the relations $E^e = F^e = 0$ and $K^e = 1$
for some integer $e\geq 2$. For this Hopf algebra, we have $G = G_{\ab} \cong \ZZ/e$ 
whereas $H_{\ab} \cong k$ if the integer~$e$ is odd and $H_{\ab} \cong k[\ZZ/2]$ if $e$ is even. 
Thus, the surjection $\varphi: k[G_{\ab}] \to H_{\ab}$ cannot be an isomorphism when $e\geq 3$.
\end{rem}

\section{\large Polynomial identities}\label{sec-PI}

A theory of polynomial identities for comodule algebras was worked out in\,\cite{AK}. 
It leads naturally to a ``universal $H$-comodule algebra''~$\UU_H$, whose definition will be recalled below.
The subalgebra of $H$-coinvariants~$\VV_H$ of~$\UU_H$ maps injectively into
the generic base algebra~$\BB_H$ defined in Section\,\ref{subsec-generic}. 
In this section we consider another localization problem for~$\BB_H$, which is motivated by the fact that 
a positive answer to it has important consequences (detailed below) for the ``versal deformation space'' 
$\AA_H = \BB_H \otimes_{\VV_H} \UU_H$.

\subsection{The universal comodule algebra}\label{subsec-UH}

Starting from a Hopf algebra~$H$, we fix another copy~$X_H$ of the underlying vector space of~$H$;
we denote the identity map from $H$ to~$X_H$ by $x\mapsto X_x$.

Consider the tensor algebra~$T(X_H)$ over the vector space~$X_H$;
it is an algebra of \emph{non-commutative polynomials}.
The algebra~$T(X_H)$ possesses a right $H$-comodule algebra structure that 
extends the natural right $H$-comodule algebra structure on~$H$:
its coaction $\delta_T: T(X_H) \to T(X_H) \otimes H$ is given by
\begin{equation*}
\delta_T(X_x)  = X_{x_1} \otimes x_2 \, . \qquad (x\in H)
\end{equation*}

We equip $S(t_H) \otimes H$ with the trivial right $H$-comodule algebra structure
whose coaction is given by~$\id \otimes \Delta$.
The subalgebra of $H$-coinvariant elements of~$S(t_H) \otimes H$ is $S(t_H) \otimes 1$, 
which we may identify with~$S(t_H)$.

By~\cite[Lemma~4.2]{AK} the algebra map $\mu : T(X_H) \rightarrow  S(t_H) \otimes H$
defined for $x\in H$ by
\begin{equation}\label{def-mu}
\mu(X_x) =  t_{x_1} \otimes x_2 
\end{equation}
is a right $H$-comodule algebra map, which is universal in the sense that any $H$-comodule algebra map
$T(X_H) \to H$ factors through~$\mu$ (see\,\cite[Th.\,4.3]{AK}).

Now let $I_H$ be the kernel of the comod\-ule algebra map
$\mu : T(X_H) \rightarrow  S(t_H) \otimes H$ defined by\,\eqref{def-mu}.
The kernel~$I_H$ is a two-sided ideal, right $H$-subcomodule of~$T(X_H)$.
Any element $P\in I_H$ is an \emph{identity for}~$H$ in the sense that it vanishes under
any right $H$-comodule algebra map $T(X_H) \rightarrow H$. 
See \cite{AK, Ka1, Ka2} for a theory of such comodule algebra identities.

Consider the quotient right $H$-comodule algebra
\begin{equation*}
\UU_H = T(X_H)/I_H \, .
\end{equation*}
We call $\UU_H$ the \emph{universal $H$-comodule algebra}
(this corresponds to the ``relatively free algebra'' in the classical literature on polynomial identities;
see\,\cite{Ro}).

By definition of~$I_H$, the map $\mu : T(X_H) \rightarrow  S(t_H) \otimes H$ induces an injection of 
right $H$-comodule algebras
\begin{equation*}\label{inject}
\bar{\mu}: \UU_H = T(X_H)/I_H \, \hookrightarrow \, S(t_H) \otimes H \, .
\end{equation*}

\subsection{The algebra of coinvariants~$\VV_H$}\label{subsec-VH}

We denote the subalgebra of right $H$-coinvariants of the universal comodule algebra~$\UU_H$
by~$\VV_H$:
\[
\VV_H =  \UU_H^{\co-H} \, .
\] 
The algebra~$\VV_H$ is a central subalgebra of~$\UU_H$.
Note that an element $\bar{P}$ of~$\UU_H$ is in~$\VV_H$ 
if and only if $\bar{\mu}(\bar{P})$ belongs to~$S(t_H) \otimes 1$.

Using Lemma\,8.1 of\,\cite{AK} and following the proof of Proposition\,9.1 of \emph{loc.\ cit.}, 
one shows that the map~$\bar{\mu}$ send~$\VV_H$ into the subalgebra~$\BB_H \cap S(t_H)$ of~$\BB_H$ 
(here we identify $S(t_H)$ with the subalgebra~$S(t_H) \otimes 1$ of~$S(t_H) \otimes H $).

By\,\cite[Cor.\,4.4]{KM}, the generic base algebra~$\BB_H$ of a Hopf algebra~$H$
(defined in Section\,\ref{subsec-generic}) has another set of generators, namely the elements
\begin{equation*}\label{def-pq}
p_x = t_{x_1} \, t_{S(x_2)} \, , \quad
q_{x,y} = t_{x_1} \, t_{y_1} \, t_{S(x_2y_2)} \, , 
\end{equation*}
\begin{equation*}\label{def-pq'}
p'_{x} = t^{-1}_{S(x_1)} \, t^{-1}_{x_2} \, , \quad
q'_{x,y} = t^{-1}_{S(x_1y_1)} \, t^{-1}_{x_2} \, t^{-1}_{y_2} \, ,
\end{equation*}
when $x,y$ run over all elements of~$H$.
The essential virtue of these generators over the older generators $\sigma(x,y)$ and $\sigma^{-1}(x,y)$
of\,\eqref{sigma-def} is that $p_x$ and $q_{x,y}$ (resp.\ $p'_x$ and $q'_{x,y}$)
are polynomials in the $t$-variables (resp. in the $t^{-1}$-variables), 
whereas the formulas for $\sigma(x,y)$ and $\sigma^{-1}(x,y)$
mix the $t$-variables and the $t^{-1}$-variables.

\begin{lemma}\label{lem-VH}
Let $H$ be a Hopf algebra.
The elements $p_x$ and $q_{x,y}$ ($x,y \in H$) belong to~$\bar{\mu}(\VV_H)$.
\end{lemma}

\begin{proof}
Consider the following elements of~$T(X_H)$:
\begin{equation*}
P_x = X_{x_1} \, X_{S(x_2)}  \, ,
\qquad
Q_{x,y} = X_{x_1} \, X_{y_1} \, X_{S(x_2 y_2)}\, ,
\end{equation*}
where $x,y \in H$.
By\,\cite[Lem\-ma\,2.1]{AK} they are coinvariant elements of $T(X_H)$,
and by\,\cite[Lemma\,4.1]{KM} we have $\mu(P_x) = p_x$ and $\mu(Q_{x,y}) = q_{x,y}$.
\end{proof}

Let $\WW_H$ be the subalgebra of~$\bar{\mu}(\VV_H)$ generated by all elements $p_x$ and $q_{x,y}$,
where $x,y$ run over~$H$.

\begin{lemma}\label{lem2-VH}
Let $H$ be a Hopf algebra.
\begin{itemize}
\item[(i)] The algebra~$\WW_H$ is a left coideal subalgebra of $S(t_H)_{\Theta}$.

\item[(ii)] The Hopf ideal $(\WW_H^+)$ of $S(t_H)_{\Theta}$ generated by $\WW_H^+ = \WW_H \cap S(t_H)_{\Theta}^+$ 
coincides with~$(\BB_H^+)$.
\end{itemize}
\end{lemma}

\begin{proof}
(i) Using the formula~\eqref{coproduct-comm-coalg} for the coproduct~$\Delta$ of~$S(t_H)_{\Theta}$, we obtain
\begin{equation*}
\Delta(p_x) = t_{x_1}t_{S(x_3)}\otimes p_{x_2} \quad\text{and}\quad
\Delta(q_{x,y}) = t_{x_1}t_{y_1}t_{S(x_3y_3)}\otimes q_{x_2,y_2} \, .
\end{equation*}
Since $\WW_H$ is generated by the elements $p_x$ and $q_{x,y}$, 
the conclusion follows.

(ii) Let $R = S(t_H)_{\Theta}/(\WW_H^+)$.
Since $\WW_H \subset \BB_H$, we have a Hopf algebra surjection $R \to S(t_H)_{\Theta}/(\BB_H^+)$. 
From
\begin{equation*}
t_{S_H(x)} = S_R(t_{x})\, ,
\quad 
t_{x}t_{y} = S_R(t_{S_H(xy)}) \in R \, ,
\end{equation*}
it follows that
\begin{equation*}
t_{x}t_{y} = S_R(S_R(t_{xy})) = t_{xy} \in R.
\end{equation*}
Here, $S_H$ (resp.\,$S_R$) denotes the antipode of $H$ (resp.\,of~$R$).  
The result for $x=y=1$, together with the existence of~$t_1^{-1}$, shows that $t_1 = 1 \in R$.
Therefore, the map $t : H \to S(t_H)_{\Theta}$ composed with the natural projection $S(t_H)_{\Theta}\to R$ 
yields a Hopf algebra surjection, which in turn factors through a Hopf algebra surjection $H_{\ab} \to R$. 
Since the composite of the last surjection with $R \to S(t_H)_{\Theta}/(\BB_H^+)$
is an isomorphism by\,\cite[Prop.\,3.1]{KM}, the desired result follows. 
\end{proof}

\subsection{The second localization problem}\label{subsec-pb2}

In the sequel we identify~$\VV_H$ with its image $\bar{\mu}(\VV_H)$ in~$\BB_H$.
We now raise another question.

\begin{Pb}[loc]
\emph{Given a Hopf algebra~$H$, 
is the algebra~$\BB_H$ a localization of the subalgebra~$\VV_H$?}
\end{Pb}

Problem\,(loc) is motivated by the fact proved in\,\cite[Sect.\,7]{AK} and in\,\cite[Sect.\,3]{KM}
that, if $\BB_H$ a localization of~$\VV_H$,
then the \emph{central localization}
\[
\AA_H = \BB_H \otimes_{\VV_H} \UU_H
\] 
of the universal comodule algebra~$\UU_H$ satisfies the following two properties.

\begin{itemize}
\item[(i)]
The extension
$\BB_H \subset \AA_H$ is a \emph{cleft $H$-Galois extension};
in particular, there is a left $\BB_H$-module, right $H$-comodule isomorphism 
\begin{equation*}
\AA_H \cong \BB_H \otimes H \, .
\end{equation*}
Thus, after localization, $\UU_H$
becomes a free module of rank~$\dim\, H$ over its subalgebra of coinvariants.

\item[(ii)]
The comodule algebra $\AA_H$ is a 
``versal deformation space'' for the \emph{forms} of~$H$ in the following sense. 
Any cleft right $H$-comodule algebra~$A$ that is a form of~$H$
(i.e., such that $k'\otimes_k A \cong k'\otimes_k H$ for some field extension $k'$ of~$k$)
is isomorphic to a comodule algebra of the form 
\begin{equation*}
\AA_H/ \mm \AA_H \, ,
\end{equation*}
where~$\mm$ is some maximal ideal of~$\BB_H$.
Conversely, if in addition $S(t_H)_{\Theta}$ is faithfully flat over~$\BB_H$,
then for any maximal ideal~$\mm$ of~$\BB_H$,
the comodule algebra~$\AA_H/ \mm \AA_H$ is a form of~$H$.
(The faithful flatness assumption is satisfied in a number of cases, 
including the case when $H$ is finite-dimensional, see\,\cite[Sect.\,3.2]{KM} and below.)
\end{itemize}

In the language of non-commutative geometry, $\AA_H$ is a 
``non-commutative fiber bundle'' over the generic base algebra~$\BB_H$.

In view of the statements of Section\,\ref{subsec-VH}, to obtain a positive answer to Problem~(loc), it suffices to check that
the elements $p'_x$ and $q'_{x,y}$ defined by\,\eqref{def-pq'} are all fractions of elements of~$\VV_H$.
This is easily verified if $x$, $y$ are certain special elements of~$H$.
Indeed, by elementary computations, if $x$, $y$ are both group-like elements, then
\begin{equation*}
p'_x = \frac{1}{p_x} \quad\text{and}\quad
q'_{x,y} = \frac{1}{q_{x,y}} \, .
\end{equation*}
Similarly, if $y$ is group-like and $x$ is skew-primitive with $\Delta(x) = g \otimes x + x \otimes h$
and $\eps(y) = 0$ for some group-like elements~$g,h$, then
\begin{equation*}
p'_x = - \frac{p_x}{p_g p_h} \quad\text{and}\quad
q'_{x,y} = - \frac{q_{x,y}}{q_{g,y} q_{h,y}} \, .
\end{equation*}

Our second main result is the following.

\begin{theorem}\label{th-local}
Let $H$ be a pointed Hopf algebra such that $S(t_H)_{\Theta}$ is faithfully flat over~$\BB_H$, 
then $\BB_H$ is a localization of~$\VV_H$.
\end{theorem}

The faithful flatness assumption is verified for instance if the pointed Hopf algebra~$H$ 
is cocommutative by\,\cite[Th.\,3.13]{KM}, 
if $k[G(H)] \hookrightarrow H$ splits as an algebra map by\,\cite[Remark 3.14\,(a)]{KM},
if $H$ is finite-dimensional, 
or if each element of the kernel of the canonical group epimorphism
$G(H)_{\ab} \to G(H_{\ab})$ is of finite order by\,\cite[Th.\,3.9]{KM}.

\begin{proof}
Since $\WW_H \subset \VV_H \subset \BB_H$, it is enough to prove that 
$\BB_H$ is a localization of~$\WW_H$.
It follows from the assumptions and from\,\cite[Lemma\,3.11]{KM} that 
\[
\BB_H= S(t_H)_{\Theta}^{\co-Q} \, , 
\]
where $Q = S(t_H)_{\Theta}/(\BB_H^+)$.
Now, it follows from Lemma\,\ref{lem2-VH}\,(ii) that 
\[
Q = S(t_H)_{\Theta}/(\WW_H^+) \,  .
\]
We conclude by applying Lemma\,\ref{lem-AB} to the Hopf algebra
$A = S(t_H)_{\Theta}$, which is pointed by Lemma\,\ref{lem-pointed},
and to the left coideal subalgebra~$B = \WW_H$.
\end{proof}

\subsection{A commutative square}\label{subsec-square}

Recall from Section\,\ref{subsec-coaction}
that $S(t_H)$ has an~$H_{\ab}$-co\-mod\-ule algebra structure 
with the coaction~$\delta_S$ defined by\,\eqref{coaction-SH}.

Let $\pi: T(X_H) \to S(t_H)$ be the abelianization map, which is the algebra map determined by
$\pi(X_x) = t_x$ for all $x\in H$.
It easily follows from the definitions that the square
\begin{equation}\label{square}
\begin{xy}
(0,7)    *+{T(X_H)}           ="a",
(33,7)   *+{S(t_H)}           ="b",
(0,-8)  *+{T(X_H) \otimes H}         ="c",
(33,-8) *+{S(t_H) \otimes H_{\ab}} ="d",
(16,0)  *{\scriptstyle\circlearrowright}="cir",
\SelectTips{eu}{}
\ar^{\pi} "a";"b"
\ar^{\delta_T}    "a";"c"
\ar^{\delta_S}   "b";"d"
\ar^{\pi\otimes q} "c";"d"
\end{xy}
\end{equation}
commutes.

From this square we deduce that
$\pi$ sends the subalgebra~$T(X_H)^{\co-H}$ of $H$-co\-in\-var\-i\-ants to the 
subalgebra~$S(t_H)^{\co-H_{\ab}}$ of $H_{\ab}$-coinvariants:
\begin{equation*}\label{coinv}
\pi: T(X_H)^{\co-H} \to S(t_H)^{\co-H_{\ab}} \, .
\end{equation*}

We can also express $(\id \otimes \, q) \circ \mu$ as the diagonal in the square:
\begin{equation}\label{eq-mu}
(\id \otimes \, q) \circ \mu 
= \delta_S \circ \pi \, .
\end{equation}

\subsection{The algebra of functions on a finite group}\label{subsec-UI}

Uma Iyer (January 2014) proved the following result, which provides a positive answer to Problem\,(loc) 
when $H$ is the algebra of functions on a finite group.

\begin{theorem}\label{thm-UI}
If $H = O_k(G)$ is the Hopf algebra of $k$-valued functions on a finite group~$G$ and 
the characteristic of~$k$ does not divide the order of~$G$, 
then $\BB_H$ is a localization of $\VV_H$.
\end{theorem}

Since $H$ is a commutative Hopf algebra, 
the canonical map $q: H \to H_{\ab}$ is the identity and $\tilde{q}: S(t_H)_{\Theta} \to H$
is given by $\tilde{q}(t_x) = x$ for all $x\in H$.

The coaction $\delta_S$ now turns $S(t_H)_{\Theta}$ and $S(t_H)$ into $H$-comodule algebras.
The commutative square~\eqref{square} means that $\pi : T(X_H) \to S(t_H)$ is an $H$-comodule algebra map
and that it induces a map on the subalgebras of coinvariant elements
\begin{equation*}
\pi: T(X_H)^{\co-H} \to S(t_H)^{\co-H} \, .
\end{equation*}
It follows from\,\eqref{eq-mu} that the map $\mu : T(X_H) \to S(t_H) \otimes H$ is given in this case by
\begin{equation}\label{eq2-mu}
\mu = \delta_S \circ \pi \, .
\end{equation}

Since $\bar{\mu}: \UU_H \to S(t_H) \otimes H$ is a comodule algebra injection 
and the subalgebra of $H$-coinvariants in~$S(t_H) \otimes H$ 
is $S(t_H) = S(t_H) \otimes 1$, we obtain the inclusion 
\[
\VV_H \subset S(t_H) \, . 
\]

\begin{lemma}\label{lem-a}
Under the hypotheses of the theorem,
the algebra $\VV_H$ is the subalgebra of $H$-coinvariants in~$S(t_H)$:
\[
\VV_H = S(t_H)^{\co-H} \, .
\]
\end{lemma}

\begin{proof}
Let $P \in T(X_H)^{\co-H}$. By definition, $\delta_T(P) = P \otimes 1$. Since $\pi$ is a comodule algebra map,
we have $\delta_S(\pi(P)) = \pi(P) \otimes 1$. 
It follows from this and from\,\eqref{eq2-mu} that 
\[
\mu(P) = \delta_S(\pi(P)) = \pi(P) \otimes 1 \, .
\]
In other words, the map $\mu$ restricted to~$T(X_H)^{\co-H}$ coincides with~$\pi$ and thus sends 
$T(X_H)^{\co-H}$ into~$S(t_H)^{\co-H}$:
\[
\mu = \pi: T(X_H)^{\co-H} \to S(t_H)^{\co-H} \, .
\]
This proves the inclusion $\VV_H \subset S(t_H)^{\co-H}$.

Under the hypotheses of the theorem, $H$ is cosemisimple. 
So the surjection $\pi: T(X_H) \to S(t_H)$ splits $H$-colinearly, 
implying that $T(X_H)^{\co-H} \to S(t_H)^{\co-H}$ is surjective.
Thus, $S(t_H)^{\co-H} = \mu(T(X_H)^{\co-H}) = \bar{\mu}(\VV_H)$.
\end{proof}

\begin{proof}[Proof of Theorem\,\ref{thm-UI}]
By Propositions\,\ref{prop-CH-loc} and\,\ref{B=C}, 
$\BB_H = \CC_H$ is a localization of the algebra~$S(t_H)^{\co-H_{\ab}} = S(t_H)^{\co-H}$.
We conclude with Lemma\,\ref{lem-a}.
\end{proof}

\section*{Acknowledgments}

We thank Uma Iyer for allowing us to include her Theorem\,\ref{thm-UI} in this paper.
We are grateful to David Harari for an enlightening remark helping us to simplify the proof of Proposition\,\ref{GNP-NP}.
The second-named author was supported by JSPS Grant-in-Aid for Scientific Research (C)~23540039.


\begin{thebibliography}{99}


\bibitem{AHN} 
E.~Aljadeff, D.~Haile, M.~Natapov,
\emph{Graded identities of matrix algebras and the universal graded algebra},
Trans.\ Amer.\ Math.\ Soc.~362 (2010), 3125--3147.

\bibitem{AK}
E.~Aljadeff, C.~Kassel,
\emph{Polynomial identities and noncommutative versal torsors},
Adv.\ Math.~218 (2008), 1453--1495.

\bibitem{Fi} E.~Fischer,
\emph{Die Isomorphie der Invariantenk\"orper der endlicher Abelschen Gruppen linearer Transformationen},
G\"ott.\ Nachr.\ 1915 (1915), 77--80. 

\bibitem{IK} U. N.~Iyer, C.~Kassel, 
\emph{Generic base algebras and universal comodule algebras for some finite-dimensional Hopf algebras}, 
Trans.\ Amer.\ Math.\ Soc., to appear; arXiv:1306.3869.

\bibitem{Ka} C.~Kassel, 
\emph{Quantum groups}, 
Graduate Texts in Mathematics,~155, 
Springer-Verlag, New York,~1995.

\bibitem{Ka1} C.~Kassel, 
\emph{Generic Hopf Galois extensions},
Quantum groups and noncommutative spaces, 104--120,
Aspects Math.,~E41, Vieweg + Teubner, Wiesbaden,~2011; arXiv:0809.0638. 

\bibitem{Ka2} C.~Kassel, 
\emph{Examples of polynomial identities distinguishing the Galois objects over finite-dimensional Hopf algebras},
Ann.\ Math.\ Blaise Pascal 20 (2013), no.~2, 175--191.

\bibitem{KM} C.~Kassel, A.~Masuoka,
\emph{Flatness and freeness properties of the generic Hopf Galois extensions},
Rev.\ Un.\ Mat.\ Argentina 51 (2010), no.~1, 79--94.

\bibitem{Ma1} A.~Masuoka, 
\emph{On Hopf algebras with cocommutative coradicals}, 
J.~Algebra 144 (1991), 451--466. 

\bibitem{Ma3} A.~Masuoka, 
\emph{Hopf cohomology vanishing via approximation by Hochschild cohomology}, 
Noncommutative geometry and quantum groups (Warsaw, 2001),
Banach Center Publ.\ 61 (2003), 111--123. 

\bibitem{Ma4} A.~Masuoka, 
\emph{Mini-course on Hopf algebras---Hopf crossed products---}, 
Conf. on Ring Theory  (Guangzhou, 2012); arXiv:1207.1532. 

\bibitem{Mo} S.~Montgomery, 
\emph{Hopf algebras and their actions on rings},
CBMS Conf.\ Series in Math., vol.~82, Amer.\ Math.\ Soc., Providence, RI, 1993.

\bibitem{No0} E.~Noether, 
\emph{Der Endlichkeitssatz der Invarianten endlicher Gruppen},
{Math.\ Ann}.\  77 (1915), no.~1, 89--92. 

\bibitem{No} E.~Noether, 
\emph{Gleichungen mit vorgeschriebener Gruppe},
Math.\ Ann.\ 78 (1917), no.~1, 221--229. 

\bibitem{Ra} D. E. Radford, 
\emph{The structure of Hopf algebras with a projection},
J.~Algebra 92 (1985), no.~2, 322--347.

\bibitem{Ro} L. Rowen, 
\emph{Polynomial identities in ring theory}, 
Pure and Applied Mathematics,~84,
Academic Press, Inc., New York-London,~1980.

\bibitem{Sa} D.~Saltman, 
\emph{Noether's problem over an algebraically closed field},
Invent.\ Math.\ 77 (1984), no.~1, 71--84. 

\bibitem{Sch} B. J. Schmid,
\emph{Finite groups and invariant theory},
Topics in invariant theory (Paris, 1989/1990), 35--66,
Lecture Notes in Math., 1478, Springer, Berlin,~1991.

\bibitem{Swa} R. G. Swan, 
\emph{Invariant rational functions and a problem of Steenrod},
Invent.\ Math.\ 7 (1969), 148--158. 

\bibitem{Swe} M.~E.\ Sweedler, 
\emph{Hopf algebras}, 
W. A. Benjamin, Inc., New York,~1969.

\bibitem{Ta}
M.~Takeuchi, 
\emph{Free Hopf algebras generated by coalgebras},
J.~Math.\ Soc.\ Japan~23 (1971), 561--582. 

\bibitem{Ta2}
M.~Takeuchi, 
\emph{Relative Hopf modules---equivalences and freeness criteria},
J.~Algebra 60 (1979), 452--471.



\end{thebibliography}
\end{document}